\documentclass[10pt, final]{amsart}
\usepackage{amsmath, amssymb,amsthm,amsfonts}
\usepackage{subfigure}
\usepackage{graphicx}
\usepackage{hyperref}
\usepackage{url}
\usepackage[numbers,sort]{natbib}
\usepackage[notcite,notref]{showkeys}
\newcommand{\abs}[1]{\left |#1\right |}
\newcommand{\norm}[1]{\left \|#1\right\|}
\newcommand{\paren}[1]{\left(#1\right)}
\newcommand{\bracket}[1]{\left[#1\right]}
\newcommand{\inner}[2]{\left\langle #1,#2\right\rangle}

\newcommand{\set}[1]{\left\{#1\right\}}
\newcommand{\mean}[1]{\left\langle#1\right\rangle}

\newcommand{\C}{\mathbb{C}}
\newcommand{\R}{\mathbb{R}}

\newcommand{\dt}{\Delta {t}}

\newcommand{\bb}{\mathbf{b}}
\newcommand{\bF}{\mathbf{F}}

\newcommand{\bu}{\mathbf{u}}
\newcommand{\bv}{\mathbf{v}}

\newcommand{\calM}{\mathcal{M}}
\newcommand{\calH}{\mathcal{H}}

\newtheorem{thm}{Theorem}[section]
\newtheorem{cor}[thm]{Corollary}
\newtheorem{lem}[thm]{Lemma}
\newtheorem{prop}[thm]{Proposition}

\newtheorem{rem}[thm]{Remark}
\renewcommand{\Im}{{\mathrm{Im}}}
\renewcommand{\Re}{{\mathrm{Re}}}
\DeclareMathOperator{\bigo}{O}

\numberwithin{equation}{section}

\title{Conservative Integrators for a  Toy Model of  Weak Turbulence}
\author{Aquil D. Jones}
\author{Gideon Simpson}
\address{Department of Mathematics
, Drexel University, Philadelphia, PA, USA}
\email{simpson@math.dexel.edu}
\author{William Wilson}
\thanks{ADJ and GS were supported by the US National
  Science Foundation grant DMS-1409018. Numerical work reported here
  was partially run on hardware supported by Drexel's University
  Research Computing Facility.  GS thanks J.L. Marzuola for
  helpful discussions on this project.}



\begin{document}

\begin{abstract}
  Weak turbulence is a phenomenon by which a system generically
  transfers energy from low to high wave numbers, while persisting for
  all finite time.  It has been conjectured by Bourgain and others
  that the 2D defocusing nonlinear Schr\"odinger equation (NLS) on the
  torus has this dynamic, and several analytical and numerical studies
  have worked towards addressing this point.

  In the process of studying the conjecture, Colliander, Keel,
  Staffilani, Takaoka, and Tao introduced a ``toy model'' dynamical
  system as an approximation of NLS, which has been subsequently
  studied numerically.  In this work, we formulate and examine several
  numerical schemes for integrating this model equation.  The model
  has two invariants, and our schemes aim to conserve at least one of
  them.  We prove convergence in some cases, and our numerical studies
  show that the schemes compare favorably to others, such as
  Trapezoidal Rule and fixed step fourth order Runge-Kutta.  The
  preservation of the invariants is particularly important in the
  study of weak turbulence as the energy transfer tends to occur on
  long time scales.
\end{abstract}

\maketitle

\section{Introduction}
\label{s:introduction}

In \cite{Colliander:2010wz}, the 2D, defocusing, cubic, toroidal
nonlinear Schr\"odinger equation (NLS),
\begin{eqnarray}
  \label{e:nls}
  i u_t + \Delta u - |u|^2 u = 0, \ \ u(0,x) = u_0(x) \ \text{for} \ x \in \mathbb{T}^2,
\end{eqnarray}
was approximated by a ``Toy Model System'' given by the equation
\begin{equation}
  \label{e:toy}
  -i\dot b_j = -\abs{b_j}^2 b_j + 2 b_{j-1}^2 \overline{b}_j
  + 2 b_{j+1}^2 \overline{b}_j,\quad \text{for $j=1\ldots N$}
\end{equation}
with boundary conditions
\begin{equation}
  \label{e:bcs}
  b_0(t) = b_{N+1}(t) = 0.
\end{equation}
Subject to these boundary conditions, \eqref{e:toy} conserves
$\ell^2$,
\begin{equation}
  \label{e:mass}
  \calM[\bb(t)] = \sum_{j=1}^N |b_j(t)|^2,
\end{equation}
and the Hamiltonian
\begin{equation}
  \label{e:eng}
  \calH[\bb(t)] = \sum_{j=1}^N \left( \frac14 | b_j(t) |^4 - \text{Re}
    ( \bar{b}_j(t)^2 b_{j-1}(t)^2) \right). 
\end{equation}
Indeed, the flow of \eqref{e:toy} can be expressed as
\begin{equation}
  \label{e:flow}
  i \dot b_j = 2 \frac{\partial \calH[\bb]}{\partial \bar b_j}.
\end{equation}
In this way, we can interpret \eqref{e:toy} as a Hamiltonian system of
nonlinearly and degenerately coupled oscillators.

\subsection{Weak Turbulence}

Roughly, $|b_j(t)|^2$ measures the spectral energy of $u$, the
solution to \eqref{e:nls}, on a set of wave numbers, $\Lambda_j$.  The
sets $\Lambda_j$ are tailored to have the property that the largest
wave number in $\Lambda_{j+1}$ exceeds the largest wave number in
$\Lambda_j$.  Thus, larger values of $|b_j|^2$ at larger values of $j$
correspond to more energy of $u$ in higher wave numbers.

The motivation for the approximation of \eqref{e:nls} by \eqref{e:toy}
was to explore a long-standing hypothesis that
\eqref{e:nls} could capture the phenomenon of weak turbulence;
see \cite{Bourgain:2004en,Dyachenko:1992bl,Colliander:2010wz} and
references therein.  Loosely speaking, a weakly turbulent system
exists globally in time, yet continuously propagates energy to ever
higher frequencies.  Thus, the norms tend to infinity, but are finite
at all finite times.  Another model equation for weak turbulence was
formulated by Majda, McLaughlin, and Tabak,
\cite{Cai:1999vx,Cai:2000kh,Cai:2001jm,Dyachenko:1992bl,Majda:1997vl,Zakharov:2001td}.

In \cite{Colliander:2010wz}, the authors proved that, given $N$, they
could construct a solution of \eqref{e:toy} which would propagate
energy from $|b_j|^2$ at low index to high index $j$.  This
corresponds to a transfer of energy in \eqref{e:nls}, and, subject to
rigorous analysis of the approximation, demonstrated that such energy
cascades were present.  However, this did not show that energy
transfer in \eqref{e:nls} was a generic phenomenon, an essential
feature of weak turbulence. Analysis of this problem has continued in
the recent works \cite{Faou:2016fo,Hani:2013kx}.  Separately, in
\cite{Colliander:2013hz,herr2013discrete}, \eqref{e:toy} was
numerically simulated and observed to have such energy transfers for a
variety of initial conditions for the lifespan of the simulations.

\subsection{Relation to Previous Work}

In \cite{Colliander:2013hz,herr2013discrete}, \eqref{e:toy} was
integrated using high order, adaptive, Runge-Kutta (RK) integrators.
While the RK integrators gave high quality results for the duration of
the simulations, they are unable to exactly conserve either of the two
invariants.\footnote{Conservation is only up to floating point error,
  which we do not consider here.}  At the same time, the energy
transfer in the toy model system is slow, requiring integration out to
long times.  Thus, the RK integrators require significant
computational effort to observe weak turbulence -- small steps are
needed for accuracy, but the phenomenological time of integration is
long.

Given the interest in \eqref{e:toy}, the goal of this work is to
present conservative methods that may aid in statistical studies of
weak turbulence and other long time integration problems.  The methods
presented here are second order in the time step, $\Delta t$.  In
numerical experiments, we observe that comparatively large time steps
can be taken with these schemes for exploring weak turbulence.  While
pointwise accuracy is lost with large $\dt$, the average energy
transfer appears robust.  In contrast, fixed time step RK schemes
will, eventually, cease to provide accurate output.

A variety of strategies have been proposed for integrating Hamiltonian
systems so as to preserve the invariants of the equations, including
splitting methods, projection methods, symplectic methods, and the
average vector field method,
\cite{Faou:2012gb,Hairer:2006vg,Celledoni:2012ff}.  In this work, we
explore some conservative discretizations of \eqref{e:toy} which
preserve either \eqref{e:mass} or \eqref{e:eng}, including implicit
midpoint.  Some appear to be {\it ad hoc} and not from one of the
aforementioned known discretization strategies, instead taking their
inspiration from known discretizations of NLS,
\cite{Matsuo:2001dc,Delfour:1981dv}.  We also direct the reader to the
recent work in \cite{Tao:2016cw}, where the author tested an explicit
symplectic integration scheme on \eqref{e:toy} with a small number of
nodes, $N=5$.

\subsection{Outline}

Our paper is organized as follows.  In Section \ref{s:integrators} we
formulate out schemes.  Next, in Section \ref{s:convergence}, we prove
a number of results on the the integrators.  Numerical simulations are
presented in Section \ref{s:numerics}, and we discuss our results in
Section \ref{s:discussion}.

\section{Conservative Integrators}
\label{s:integrators}

In this section, we formulate our discretizations and prove that they
conserve the relevant invariant, under a solvability assumption.
These schemes are all symmetric, but the nonlinear terms are treated
differently in each case.  Since the dependent variables of
\eqref{e:toy} are nonlinearly and degenerately coupled, formulating a
conservative scheme is nontrivial.  This is contrast to NLS, where the
spatial coupling, the analog of the lattice site coupling of
\eqref{e:toy}, is linear.

Before stating the schemes, we introduce some notation.  Throughout,
we assume the time step, $\dt$, is constant.  We will solve at times
$t_n = n \Delta t$, and our approximation will be
$\bb_n \approx \bb(t_n)$, with components $b_{j,n} \approx b_j(t_n)$.
We define the following linear and nonlinear averages as:
\begin{subequations}
  \label{e:notation}
  \begin{align}
    b_{j,n} &\approx b_j(t_n)\\
    b_{j,n+1/2} & = \tfrac{1}{2}\paren{  b_{j,n} + b_{j,n+1}}\\
    b_{j,n+1/2}^2 & = \bracket{\tfrac{1}{2}\paren{  b_{j,n} + b_{j,n+1}}}^2\\
    (b^2)_{j,n+1/2} & = \tfrac{1}{2}\paren{  b_{j,n}^2 + b_{j,n+1}^2}\\
    \abs{b}^2_{j,n+1/2} &= \tfrac{1}{2} \paren{\abs{b_{j,n}}^2 + \abs{b_{j,n+1}}^2  }
  \end{align}
\end{subequations}
Using this notation, the trapezoidal method corresponds to
\begin{equation}
  \label{e:trap}
  \begin{split}
    {\bf b}_{n+1} &= {\bf b}_n + \dt {\bf F}^{\rm Trap}(\bb_n,
    \bb_{n+1};\dt)\\
    F_j^{\rm Trap} & =\frac{i}{2}\left[-\abs{b_{j,n}}^2 b_{j,n} + 2
      \overline{b}_{j,n}(b_{j-1,n}^2 +b_{j+1,n}^2 )\right]\\
    &\quad + \frac{i}{2}\left[-\abs{b_{j,n+1}}^2 b_{j,n+1} + 2
      \overline{b}_{j,n+1}(b_{j-1,n+1}^2 + b_{j+1,n+1}^2 )\right]
  \end{split}
\end{equation}
while the implicit midpoint method corresponds to
\begin{equation}
  \label{e:midpoint}
  \begin{split}
    {\bf b}_{n+1} &= {\bf b}_n + \dt {\bf F}^{\rm Midpt}(\bb_n,
    \bb_{n+1};\dt)\\
    F_j^{\rm Midpt} & =- i|b_{j,n+1/2}|^2b_{j,n+1/2} + 2 i\bar
    b_{j,n+1/2} (b_{j-1,n+1/2}^2 + b_{j+1,n+1/2}^2)
  \end{split}
\end{equation}

While the results we present in this section are for the Dirichlet
type boundary conditions \eqref{e:bcs}, it can be verified that they
also apply to the periodic boundary condition case,
\begin{equation}
  \label{e:periodic}
  b_0(t) = b_{N}(t), \quad b_{N+1}(t) = b_1(t).
\end{equation}

\subsection{Mass Preserving Integrators}
\label{s:mass_integrator}

We first note that the implicit midpoint method, given by
\eqref{e:midpoint}, conserves \eqref{e:mass}:
\begin{prop}
  \label{p:midpoint}
  Assume that the nonlinear algebraic system in \eqref{e:midpoint} can
  be solved exactly for $b_{j,n+1}$.  Then \eqref{e:mass} is exactly
  conserved.
\end{prop}
\begin{proof}
  This is a consequence of the mass invariant being quadratic which
  can be expressed as
  \begin{equation*}
    \mathcal{M}[\bb] = \inner{\bb}{\bb} = \Re \sum_{j=1}^N (b_j \bar b_j).
  \end{equation*}
  By virtue of conservation of mass,
  \begin{equation*}
    \inner{\bb}{\frac{d\bb}{dt}} = \Re \sum_{j=1}^N \bar b_j \frac{d b_j}{dt} = 0.
  \end{equation*}
  The conservation for the scheme, being symplectic, then follows from
  a result due to Cooper, \cite{Cooper:1987ea,Hairer:2006vg}.
\end{proof}

The following modified midpoint method also preserves mass:
\begin{equation}
  \label{e:mass_scheme}
  \begin{split}
    9    \bb_{n+1} & = \bb_n + \dt {\bf F}^{\rm Mass}(\bb_n, \bb_{n+1};\dt)\\
    F_{j}^{\rm Mass}& = -i\abs{b}^2_{j,n+1/2} b_{j,n+1/2} + 2 i\bar
    b_{j,n+1/2}\paren{ b_{j+1,n+1/2}^2 +b_{j-1,n+1/2}^2 }
  \end{split}
\end{equation}
Note the difference in the self-interaction schemes of
\eqref{e:mass_scheme} and \eqref{e:midpoint},
\[
  - \abs{b}^2_{j,n+1/2} b_{j,n+1/2} \quad \text{vs} \quad
  -|b_{j,n+1/2}|^2b_{j,n+1/2}.
\]
This treatment of the nonlinear term in \eqref{e:mass_scheme} is
modeled upon the analogous expression for some conservative NLS
schemes, \cite{Matsuo:2001dc,Delfour:1981dv}.

\begin{prop}
  \label{p:mass}
  Assume that the nonlinear algebraic system in \eqref{e:mass_scheme}
  can be solved exactly for $b_{j,n+1}$.  Then \eqref{e:mass} is
  exactly conserved.
\end{prop}

\begin{proof}
  Multiplying the $j$- th equation of\eqref{e:mass_scheme} by
  $\bar b_{j,n+1/2}$, we have
  \begin{equation*}
    \begin{split}
      &{\abs{b_{j,n+1}}^2 - \abs{b_{j,n}}^2 - b_{j,n} \bar b_{j,n+1} +
        \bar
        b_{j,n} b_{j,n+1}} \\
      &\quad = -2 i \Delta t \abs{b}_{j, n+1/2}^2\abs{b_{j, n+1/2}}^2
      + 4 i \Delta t \bar b_{j,n+1/2}^2 \paren{ b_{j+1,n+1/2}^2 +
        b_{j-1,n+1/2}^2}
    \end{split}
  \end{equation*}
  Taking the real part of this equation, we obtain
  \begin{equation*}
    \begin{split}
      &\abs{b_{j,n+1}}^2 - \abs{b_{j,n}}^2 \\
      &\quad =4\Delta t \Re \set{ i \bar b_{j,n+1/2}^2 b_{j+1,n+1/2}^2
        + i \bar b_{j,n+1/2}^2 b_{j-1,n+1/2}^2 },
    \end{split}
  \end{equation*}
  and summing over $j$,
  \begin{equation*}
    \begin{split}
      &\sum_{j=1}^{N}\abs{b_{j,n+1}}^2 - \abs{b_{j,n}}^2  \\
      &\quad= 4\Delta t \sum_{j=1}^N \Re \set{ i \bar b_{j,n+1/2}^2
        b_{j+1,n+1/2}^2 + i \bar b_{j,n+1/2}^2 b_{j-1,n+1/2}^2 }.
    \end{split}
  \end{equation*}
  Shifting indices on the second term and using that
  $b_0 =b_{N+1} = 0$, we have
  \begin{equation*}
    \begin{split}
      &\sum_{j=1}^N \Re \set{ i \bar b_{j,n+1/2}^2 b_{j+1,n+1/2}^2 + i
        \bar
        b_{j,n+1/2}^2  b_{j-1,n+1/2}^2 }\\
      &\quad =\sum_{j=1}^N \Re \set{ i \bar b_{j,n+1/2}^2
        b_{j+1,n+1/2}^2 + i \bar
        b_{j+1,n+1/2}^2  b_{j,n+1/2}^2 }\\
      &\quad = \sum_{j=1}^N \Re \set{i 2\Re \paren{\bar b_{j,n+1/2}^2
          b_{j+1,n+1/2}^2}} = 0.
    \end{split}
  \end{equation*}
  Therefore,
  \[
    \sum_{j=1}^N \abs{b_{j,n+1}}^2 =\sum_{j=1}^N \abs{b_{j,n}}^2.
  \]

\end{proof}

\subsection{Hamiltonian Preserving Integrator}
\label{s:neg_integrator}

One numerical scheme which exactly preserves the energy,
\eqref{e:eng}, is:
\begin{equation}
  \label{e:eng_scheme}
  \begin{split}
    \bb_{n+1} & = \bb_n + \dt \bF^{\rm Eng}(\bb_n, \bb_{n+1};\dt)\\
    F^{\rm Eng}_j & = - i\abs{b}^2_{j,n+1/2} b_{j,n+1/2} +2i \bar
    b_{j,n+1/2}\bracket{ (b^2)_{j+1,n+1/2} +(b^2)_{j-1,n+1/2}}
  \end{split}
\end{equation}
Note how the $j$-$j+1$ and $j$-$j-1$ interaction terms are handled
differently than in \eqref{e:trap}, \eqref{e:midpoint}, and
\eqref{e:mass_scheme}.  This scheme, which we show to exactly preserve
\eqref{e:eng}, appears to be {\it ad hoc}, and is not based upon a
known methodology, such as average vector field, for deriving schemes
which conserve the energy.

\begin{prop}
  Assume that the nonlinear algebraic system in \eqref{e:eng_scheme}
  can be solved exactly for $b_{j,n+1}$.  Then \eqref{e:eng} is
  exactly conserved.
\end{prop}

\begin{proof}
  Multiplying the $j$-th equation of \eqref{e:eng_scheme} by
  $\bar b_{j,n+1} - \bar b_{j,n}$, we have
  \begin{equation*}
    \begin{split}
      \abs{ b_{j,n+1} - b_{j,n}}^2 &= -i \Delta t \abs{b}^2_{j,n+1/2}
      b_{j,n+1/2} (\bar b_{j,n+1} - \bar b_{j,n})\\
      & \quad + 2 i \Delta t \bar b_{j,n+1/2}\bracket{
        (b^2)_{j+1,n+1/2} +
        (b^2)_{j-1,n+1/2}}(\bar b_{j,n+1} - \bar b_{j,n})\\
      & = -i \frac{\Delta t }{4} (\abs{b_{j,n+1}}^4 - \abs{b_{j,n}}^4
      ) +\Delta t\abs{b}^2_{j,n+1/2} \Im(
      b_{j,n}\bar b_{j,n+1})\\
      & \quad + i \Delta t \bracket{ (b^2)_{j+1,n+1/2}
        +(b^2)_{j-1,n+1/2} }\bracket{\bar b^2_{j,n+1} - \bar b^2_{j,n}
      }
    \end{split}
  \end{equation*}
  Taking the imaginary parts and dividing out by $\Delta t$,
  \begin{equation}
    \label{e:im1_eng}
    \begin{split}
      0 &= \tfrac{1 }{4}  (\abs{b_{j,n+1}}^4 - \abs{b_{j,n}}^4 )  \\
      &\quad -\Im \set{i \bracket{ (b^2)_{j+1,n+1/2}
          +(b^2)_{j-1,n+1/2} }\bracket{\bar b^2_{j,n+1} - \bar
          b^2_{j,n} }}
    \end{split}
  \end{equation}
  Examining the second term,
  \begin{equation*}
    \begin{split}
      &\bracket{ (b^2)_{j+1,n+1/2} +(b^2)_{j-1,n+1/2}}\bracket{\bar
        b^2_{j,n+1} - \bar b^2_{j,n} } \\
      &=\frac{1}{2}\left [ b^2_{j+1,n+1} \bar b^2_{j,n+1} -
        b^2_{j+1,n+1} \bar b^2_{j,n} + b^2_{j+1,n} \bar b^2_{j,n+1} -
        b^2_{j+1,n}
        \bar b^2_{j,n}  \right .\\
      &\quad \left.+ b^2_{j-1,n+1} \bar b^2_{j,n+1} - b^2_{j-1,n+1}
        \bar b^2_{j,n} + b^2_{j-1,n} \bar b^2_{j,n+1} - b^2_{j-1,n}
        \bar b^2_{j,n} \right ]
    \end{split}
  \end{equation*}
  Summing from $j=1,\ldots, N$, and using $b_0 = b_{N+1} = 0$ to shift
  indices,
  \begin{equation*}
    \begin{split}
      &\sum_{j=1}^N \frac{1}{2}\left [b^2_{j+1,n+1} \bar b^2_{j,n+1} -
        b^2_{j+1,n+1} \bar b^2_{j,n} + b^2_{j+1,n} \bar b^2_{j,n+1} -
        b^2_{j+1,n}
        \bar b^2_{j,n} \right.  \\
      &\quad \left. + b^2_{j-1,n+1} \bar b^2_{j,n+1} - b^2_{j-1,n+1}
        \bar b^2_{j,n} + b^2_{j-1,n} \bar b^2_{j,n+1} - b^2_{j-1,n}
        \bar b^2_{j,n} \right]\\
      & = \sum_{j=1}^N \frac{1}{2}\left[ b^2_{j+1,n+1} \bar
        b^2_{j,n+1} - b^2_{j+1,n+1} \bar b^2_{j,n} + b^2_{j+1,n} \bar
        b^2_{j,n+1} - b^2_{j+1,n}
        \bar b^2_{j,n} \right. \\
      &\quad \left.  + b^2_{j,n+1} \bar b^2_{j+1,n+1} - b^2_{j,n+1}
        \bar b^2_{j+1,n} + b^2_{j,n} \bar b^2_{j+1,n+1} - b^2_{j,n}
        \bar b^2_{j+1,n} \right]\\
      & = \sum_{j=1} \Re \bracket{b^2_{j+1,n+1} \bar b^2_{j,n+1} -
        b^2_{j+1,n} \bar b^2_{j,n} } + i \Im \bracket{ b^2_{j,n} \bar
        b^2_{j+1,n+1} + b^2_{j+1,n} \bar b^2_{j,n+1} }
    \end{split}
  \end{equation*}
  Summing \eqref{e:im1_eng} over $j$, and using the preceding
  calculation,
  \begin{equation*}
    0 = \sum_{j=1}^N \frac{1 }{4}  (\abs{b_{j,n+1}}^4 - \abs{b_{j,n}}^4 )
    - \sum_{j=1}^N  \Re  \bracket{b^2_{j+1,n+1} \bar b^2_{j,n+1} - b^2_{j+1,n}
      \bar b^2_{j,n}  }
  \end{equation*}
  Therefore,
  \begin{equation*}
    \sum_j \frac{1 }{4} \abs{b_{j,n+1}}^4  - \Re (b^2_{j+1,n+1} \bar
    b^2_{j,n+1} ) = \sum_j \frac{1 }{4} \abs{b_{j,n}}^4 - \Re ( b^2_{j+1,n}
    \bar b^2_{j,n})
  \end{equation*}
  Shifting indices and using $b_0 = b_{N+1} = 0$ again yields the
  conservation of the energy, \eqref{e:eng}, in the discretized
  evolution.

  Again, a similar calculation holds in the case of periodic boundary
  conditions, $b_0 = b_{N}$ and $b_{N+1} = b_1$.
\end{proof}

\section{Convergence of Integrators}
\label{s:convergence}
In this section we examine the convergence of our integrators.  We
provide a complete proof in the case of mass preserving schemes, as it
gives {\it a priori} estimates on $b_{j,n}$ at all $j$ and $n$. We can
also provide a partial proof for the energy preserving scheme.  We
proceed in the following steps.  First we prove results on the
solvability of the nonlinear algebraic systems.  Next we establish
stability and consistency. Finally, we prove convergence.

\subsection{Solvability}

We prove solvability via the implicit function theorem.

\begin{lem}
  \label{l:solvability}
  Given $\bb_n$ and one of the three schemes, there exists $\dt_1>0$,
  depending on $\|\bb_n\|_2$, such that for all $\dt \leq \dt_1$, a
  unique solution, $\bb_{n+1}$, exists.  Moreover, as a function of
  $\dt$, $\bb_{n+1}$ is $C^1$.
\end{lem}
\begin{proof}
  We frame this in terms of real and imaginary parts, with
  $\bb = \bu + i \bv$, and $\bu, \bv\in \R^{N}$.  Define the function
  $\mathcal{F}:\R^{2N}\times \R\to \R^{2N}$ as follows.
  \begin{equation}
    \mathcal{F}(\bu,\bv, \dt; \bb_n) = \begin{pmatrix} \bu \\
      \bv \end{pmatrix} - \begin{pmatrix} \Re\bb_n \\
      \Im\bb_n \end{pmatrix}  - \dt \begin{pmatrix}
      \Re\bF(\bb_n,\bu + i \bv;\dt)\\
      \Im\bF(\bb_n,\bu + i \bv;\dt)
    \end{pmatrix}
  \end{equation}
  Notice that $\Re\bF(\bb_n,\bu + i \bv;\dt)$ and
  $\Im\bF(\bb_n,\bu + i \bv;\dt)$ are cubic in the components of $\bu$
  and $\bv$.  Taking a gradient of the above expression in
  $(\bu, \bv)$,
  \begin{equation}
    \nabla_{\bu,\bv}\mathcal{F} = I - \dt \nabla_{\bu, \bv}\begin{pmatrix}
      \Re\bF(\bb_n,\bu + i \bv;\dt)\\
      \Im\bF(\bb_n,\bu + i \bv;\dt)
    \end{pmatrix}
  \end{equation}
  At $\dt = 0$,the solution is simply $\bu = \Re \bb_n$ and
  $\bv =\Im \bb_n$.  Furthermore,
  \begin{equation*}
    \left.\nabla_{\bu, \bv} \mathcal{F} \right|_{\substack{\bu =
        \bu_n\\\bv=\bv_n\\\dt=0}} = I.
  \end{equation*}
  Therefore the implicit function theorem applies and there exists a
  positive interval of $\dt$ values for which we can compute
  $\dt \mapsto (\bu, \bv)\mapsto \bb = \bu + i \bv$.

  That $\dt_1$ is controlled by $\|\bb_{n}\|_2$ follows from our use
  of the $\ell^2$ topology in our application of the implicit function
  theorem, \cite{rudin1976principles}.

\end{proof}

\begin{cor}
  \label{c:mass}
  Given an initial condition, $\bb_0$, there exists a $\dt_1$ such
  that for any fixed $\dt\leq \dt_1$, the mass preserving schemes can
  always be solved.
\end{cor}
\begin{proof}
  By Lemma \ref{l:solvability}, we can find a $\dt_1$ to compute
  $\bb_1$ for an admissible $\dt$.  Since the mass schemes conserve
  $\ell_2$, $\bb_1$ has the same $\ell_2$ norm as $\bb_0$.  Since
  $\dt_1$ only depends on the magnitude of this norm, it remains an
  applicable value for computing $\bb_2$, which can be computed at the
  same $\dt$.  By induction this can be carried on to any iterate.
\end{proof}

\begin{rem}
  Since the energy invariant \eqref{e:eng} does not provide {\it a
    priori} bounds on a norm, we are unable to show global persistence
  of the energy preserving scheme.
\end{rem}

\subsection{Stability}
Obviously, our schemes have the desired stability property owing to
the smoothness of the $\bF$ functions in all cases:
\begin{lem}
  \label{l:stability}
  For each integration scheme, there exists a polynomial $\Pi$, with
  positive coefficients, such that for all
  $\bu, \bv, \tilde \bu,\tilde\bv \in \C^N$:
  \[
    \|\bF(\bu, \bv) - \bF(\tilde\bu, \tilde\bv)\| \leq \Pi(\|\bu\|_2,
    \|\bv\|_2, \|\tilde\bu\|_2, \|\tilde\bv\|_2) (\|\bu - \tilde\bu\|
    + \|\bv - \tilde\bv\|)
  \]
\end{lem}
\begin{proof}
  Since the components of $\bF$ are cubic in their arguments and the
  $\ell^2$ norm gives pointwise control, the result is immediate.
  Each scheme will have a different polynomial.

\end{proof}

\subsection{Consistency}

Before obtaining the consistency result, we state a lemma about the
solution to \eqref{e:toy}.
\begin{thm}
  \label{t:toy}
  Let $\bb(t)$ be the solution to \eqref{e:toy} for initial condition
  $\bb(0)$.  Then $\bb(t)$ is a $C^\infty$ and a global in time such
  that for any $k$, there is a polynomial, $\Pi_k$, with positive
  coefficients such that
  \begin{equation}
    \norm{\frac{d^k \bb}{dt^k }}_2 \leq \Pi_{k}(\norm{\bb(0)}_{2})
  \end{equation}
\end{thm}
\begin{proof}
  Since \eqref{e:toy} has a righthand side which is polynomial in
  $b_j$ and $\bar b_j$, a $C^1$ local in time solution exists.  Since
  it conserves $\ell^2$, it will in fact be global.  By a bootstrap
  argument, it will also be $C^\infty$.  The polynomial bound in terms
  of the data can then be obtained by induction.
\end{proof}

\begin{lem}
  \label{l:consistency}
  For any of the conservative schemes, the local truncation error is
  \begin{equation*}
    \|\bb(t_{n+1}) - \bb(t_{n}) - \dt \bF(\bb(t_n), \bb(t_{n+1})) \| \leq
    C \dt^3,
  \end{equation*}
  and the constant $C$ only depends on $\|\bb(0)\|_2$.
\end{lem}

\begin{proof}
  We prove this in the case of the mass conserving scheme
  \eqref{e:mass_scheme}.  The proofs for implicit midpoint and the
  energy preserving scheme are similar.

  Substituting $\bb(t)$ into \eqref{e:mass_scheme},
  \begin{equation}
    \label{e:consistency1}
    \begin{split}
      \tfrac{i}{\Delta t }(b_{j}(t_{n+1}) - b_{j}(t_n)) -
      &\tfrac{1}{4}(|b_j(t_n)|^2 + |b_j(t_{n+1})|^2) (b_j(t_n) + b_j(t_{n+1})) \\
      \quad + \tfrac{1}{4} (\bar b_j(t_n) + \bar
      b_j(t_{n+1}))&\left[(b_{j+1}(t_n) + b_{j+1}(t_{n+1}))^2 \right.\\
      &\quad+ \left.(b_{j-1}(t_n) + b_{j-1}(t_{n+1}))^2 \right],
    \end{split}
  \end{equation}
  and Taylor expanding about $\dt = 0$,
  \begin{align}
    &\tfrac{i}{\Delta t }(b_{j}(t_{n+1}) - b_{j}(t_n))   = i \dot b_j +
      \tfrac{i}{2}\ddot b_j \dt +\bigo(\dt^2)\\
    \begin{split}
      &-\tfrac{1}{4}(|b_j(t_n)|^2 + |b_j(t_{n+1})|^2) (b_j(t_n) +
      b_j(t_{n+1})) \\
      &= - |b_j|^2 b_j + \dt \bracket{-\tfrac{1}{2}(\bar b_j \dot b_j
        + b_j\dot{\bar{b}}_j)b_j- \tfrac{1}{2}|b_j|^2 \dot b_j} +
      \bigo(\dt^2)
    \end{split}\\
    \begin{split}
      &\tfrac{1}{4} (\bar b_j(t_n) + \bar
      b_j(t_{n+1}))\bracket{(b_{j+1}(t_n) + b_{j+1}(t_{n+1}))^2 +
        (b_{j-1}(t_n) + b_{j-1}(t_{n+1}))^2  }\\
      &=2 \bar b_j(b_{j+1}^2 + b_{j-1}^2) + \dt
      \bracket{\dot{\bar{b}}_j(b_{j+1}^2 + b_{j-1})^2 + 2 \bar
        b_j(b_{j+1}\dot{b}_{j+1} + b_{j-1} \dot{b}_{j-1})} \\
      &\quad + \bigo(\dt^2)
    \end{split}
  \end{align}
  In the above three expressions, we have suppressed the $t_n$
  dependence in the terms on the righthand side.  The $\bigo(\dt^2)$
  expressions reflect the remainder terms from the Taylor expansion,
  which are {\it a priori} bounded by Theorem \ref{t:toy}.
  Substituting back into \eqref{e:consistency1}, and writing the
  equation in terms of $\bb(t_{n+1}) - \bb(t_n)$ yields the result.

\end{proof}


\subsection{Convergence}

The above consistency and stability results allow us to prove, for the
mass preserving schemes:

\begin{thm}
  \label{t:mass_conv}

  Given an initial condition $\bb(0)$, for either of the mass
  conserving schemes, there exists a $\dt_\star$, $K>0$ and $C>0$,
  such that for all $\dt \leq \dt_\star$ and all $n$, the error is
  \begin{equation}
    \|\bb_{n} - \bb(t_n)\|_2 \leq C (e^{K t_n}-1)\dt^2
  \end{equation}

\end{thm}

\begin{proof}
  As the proof is standard, we omit the details.  We note that our
  above result holds for any $n$ by virtue of the conservation of
  $\ell^2$, which gives uniform in $n$ estimates of all constants from
  the stability and consistency estimates.

\end{proof}

\begin{rem}
  If an {\it a priori} bound were available for the energy preserving
  scheme, the analog of Theorem \ref{t:mass_conv} would hold for it.
\end{rem}

\section{Numerical Results}
\label{s:numerics}

In this section, we explore our methods, and compare them to others.
Throughout, we make use of the Portable, Extensible Tookit for
Scientific Computation (PETSc),
\cite{petsc-efficient,petsc-user-ref,petsc-web-page}, which contains
Crank-Nicolson (Trapezoidal Rule), Implicit Midpoint, and
Runge-Kutta 4 (RK4) as subroutines.  As both linear and nonlinear
solvers are required, for these methods, we compute with:
\begin{itemize}

\item An absolute tolerance of $10^{-50}$, such that the nonlinear
  solver terminates at step $k$ if the norm of the residual,
  $\|\mathbf{r}^{(k)}\|$, falls beneath this value.

\item A relative tolerance of $10^{-15}$, such that the nonlinear
  solver terminates at step $k$ if the norm of the residual relative
  to the initial residual,
  $\|\mathbf{r}^{(k)}\|/\|\mathbf{r}^{(0)}\|$, falls beneath this
  value.

\item A step tolerance of $10^{-15}$ such that the nonlinear solver
  terminates at step $k$ if the norm of the step relative to the norm
  of the approximate solution,
  $\|\Delta\mathbf{b}^{(k)}\|/\|\mathbf{b}^{(k)}\|$, falls beneath
  this value.
\end{itemize}
These settings mitigate the error due to the nonlinear solvers
allowing for a direct examination of the impact of $\dt$.  The
nonlinear solvers typically terminate due to either the relative norm
or the step size becoming small.

We also compare against a Projection method,
\begin{equation}
  \bb_{n+1} = P(\underbrace{\bb_n + \dt \bF^{\rm
      RK4}(\bb_n;\dt)}_{\equiv \bb_{n+1}^{\rm RK4}}; \calM_0, \calH_0).
\end{equation}
Here, a candidate solution for time step $n+1$ is produced with RK4,
and the projector finds an element of
$\set{\bb\in \C^N\mid \calM[\bb]=\calM_0, \quad \calH[\bb]=\calH_0}$,
closest to ${\bb_{n+1}^{\rm RK4}}$ with respect to the 2-norm.
Following, \cite{Hairer:2006vg}, this is approximated by finding
Lagrange multipliers $\lambda_\calM$ and $\lambda_\calH$ as roots of
the function
\begin{equation}
  \begin{split}
    &g(\lambda_\calM,\lambda_\calH; \bb_{n+1}^{\rm RK4}, \calM_0,
    \calH_0 )
    \\
    &= \begin{pmatrix} \calM[\bb_{n+1}^{\rm RK4} +
      \lambda_\calM\nabla_\bb\calM[\bb_{n+1}^{\rm RK4} ] +
      \lambda_\calH\nabla_\bb\calH[\bb_{n+1}^{\rm
        RK4} ]]-\calM_0\\
      \calH[\bb_{n+1}^{\rm RK4} +
      \lambda_\calM\nabla_\bb\calM[\bb_{n+1}^{\rm RK4} ] +
      \lambda_\calH\nabla_\bb\calH[\bb_{n+1}^{\rm RK4} ]]-\calH_0
    \end{pmatrix}.
  \end{split}
\end{equation}
The projected solution,
\begin{equation}
  \bb_{n+1} = \bb_{n+1}^{\rm RK4} + \lambda_\calM\nabla_\bb\calM[\bb_{n+1}^{\rm
    RK4} ] + \lambda_\calH\nabla_\bb\calH[\bb_{n+1}^{\rm
    RK4} ],
\end{equation}
should then conserve both invariants.  The root finding problem is now
in $\R^2$ instead of $\C^{N}$.  Here, we take an absolute tolerance of
$10^{-12}$, relative tolerance of $10^{-15}$, and a step tolerance of
$10^{-15}$.  The Projection method solver typically terminates because
it satisfies the absolute tolerance criterion.

\subsection{Pointwise Convergence}
\label{s:pointwise}
For an assessment of pointwise convergence, we use, as an initial
condition,
\begin{equation}
  \label{e:ic_shock}
  b_j = e^{i (j-1) \pi/4},\quad j = 1, \ldots, N.
\end{equation}
The evolution of this data, which was previously studied in
\cite{Colliander:2013hz} due to its interesting dynamics, is shown
here in Figure \ref{f:shock1} for $N=100$.  The notable feature of
this solution is the rarefactive wave behavior in the left half of the
domain and dispersive shock wave behavior in the right half.  This was
more extensively studied in \cite{herr2013discrete}.

Integrating the system out to $t_{\max} = 5$ for different values of
$\dt$, convergence results appear in Figure \ref{f:convergence}.
Here, the error is measured as
\begin{equation}
  \label{e:error1}
  \max\limits_{n}\|
  \bb^{\rm Scheme}_n- \bb^{\rm RK4}(t_n)\|_{2},
\end{equation}
where $\bb^{\rm Scheme}$ is one of the schemes, and $\bb^{\rm RK4}$ is
the RK4 solution computed with $\Delta t = 10^{-4}$.  This RK4
solution serves as a surrogate for the true solution.  As expected, we see $\bigo(\dt^2)$ convergence for the
conservative schemes.

A more thorough comparison of the integrators is given in
Tables \ref{t:err1}, \ref{t:mass1}, and \ref{t:eng1}, where, in each
case, the problem is integrated out to $t_{\max}=1$. The error,
\eqref{e:error1}, is comparable amongst \eqref{e:mass_scheme},
\eqref{e:eng_scheme}, Implicit Midpoint, and Trapezoidal Rule.  The
error is roughly an order of magnitude smaller for the Projection
method.  The invariants are preserved in the expected cases and
otherwise show $\bigo(\dt^2)$ convergence.  We note that the relative error in
the invariants is worst for Trapezoidal Rule, which conserves neither
invariant.  

\begin{figure}

  \subfigure{\includegraphics[width=6.25cm]{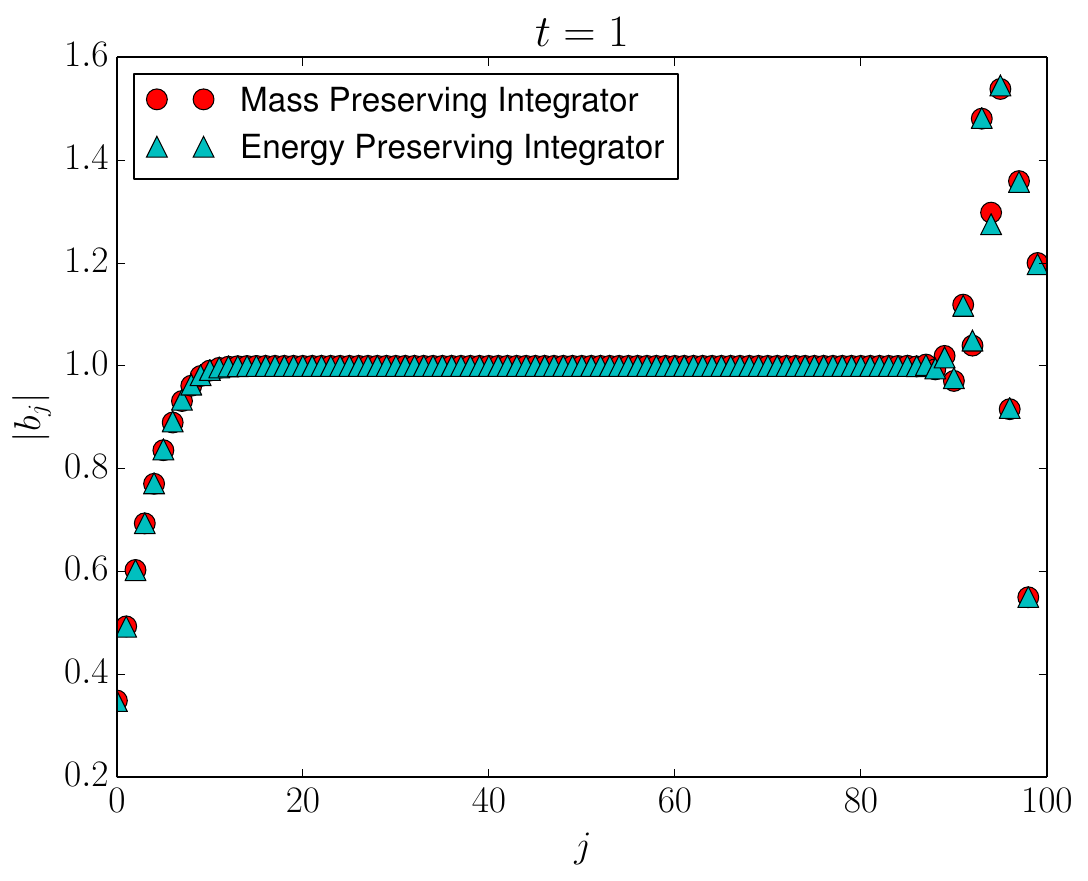}}
  \subfigure{\includegraphics[width=6.25cm]{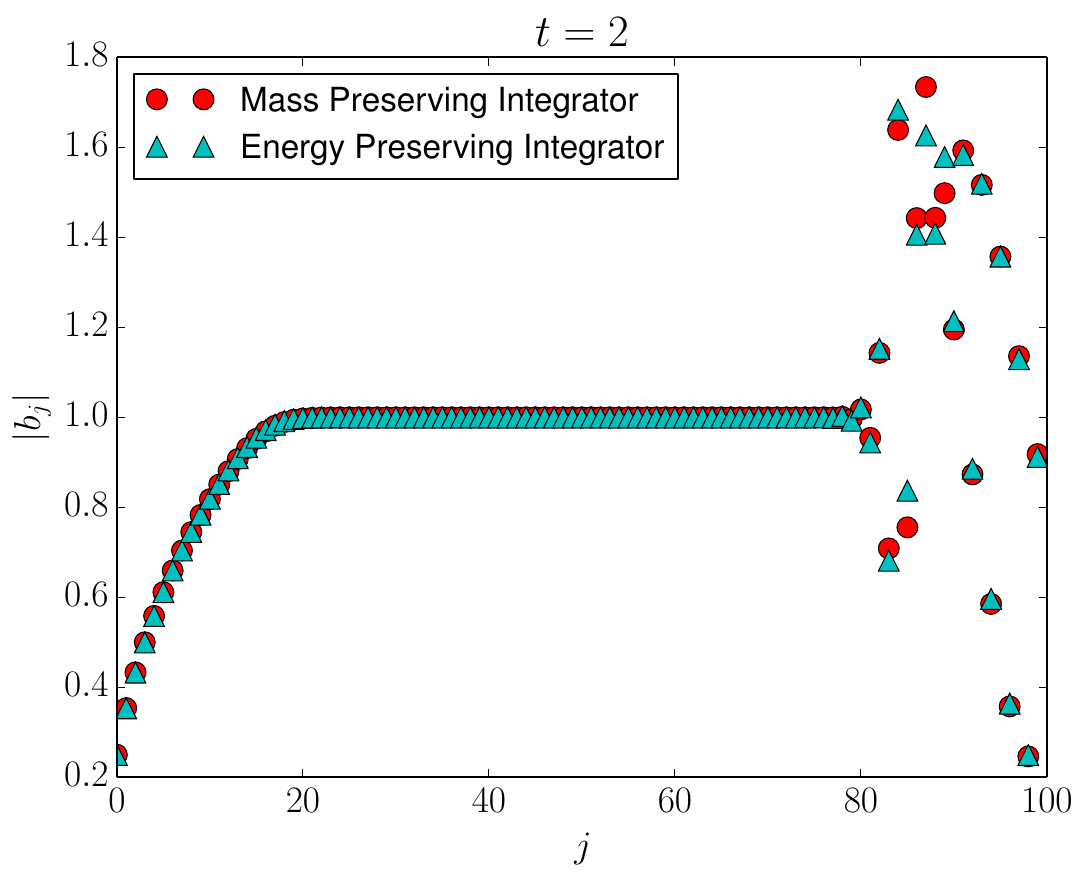}}

  \subfigure{\includegraphics[width=6.25cm]{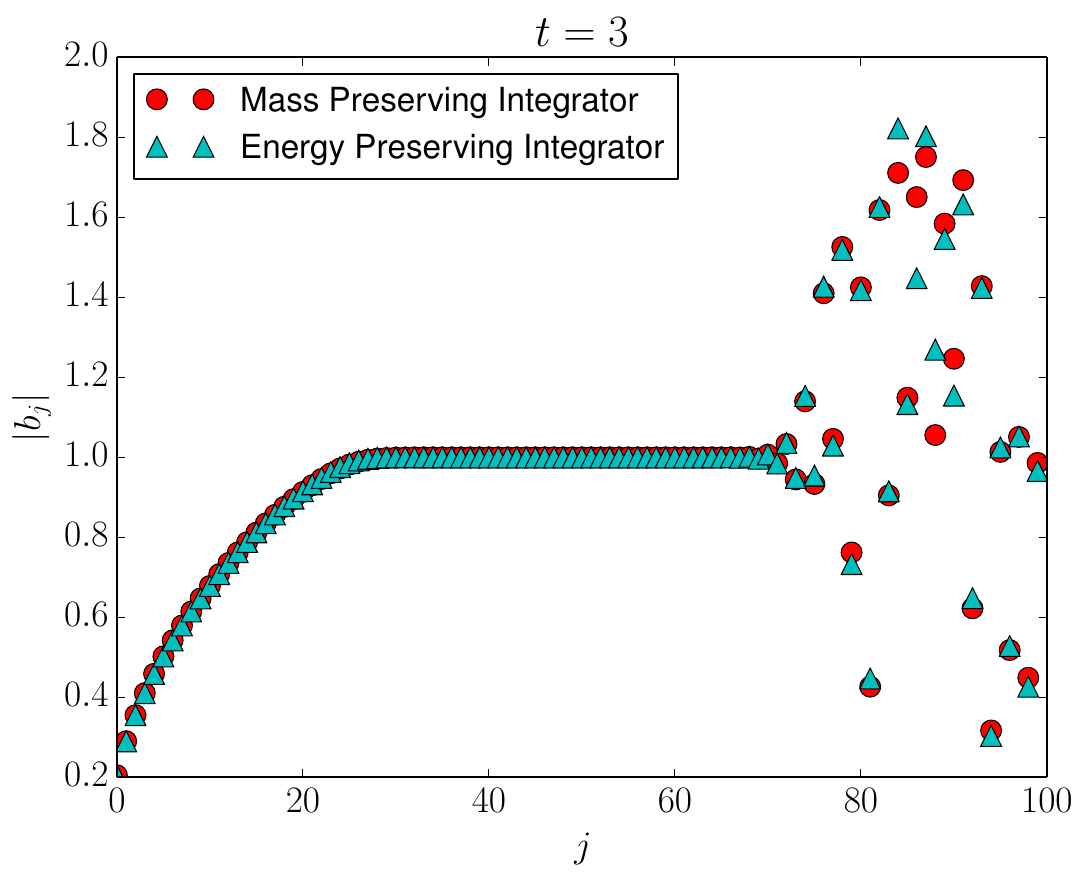}}
  \subfigure{\includegraphics[width=6.25cm]{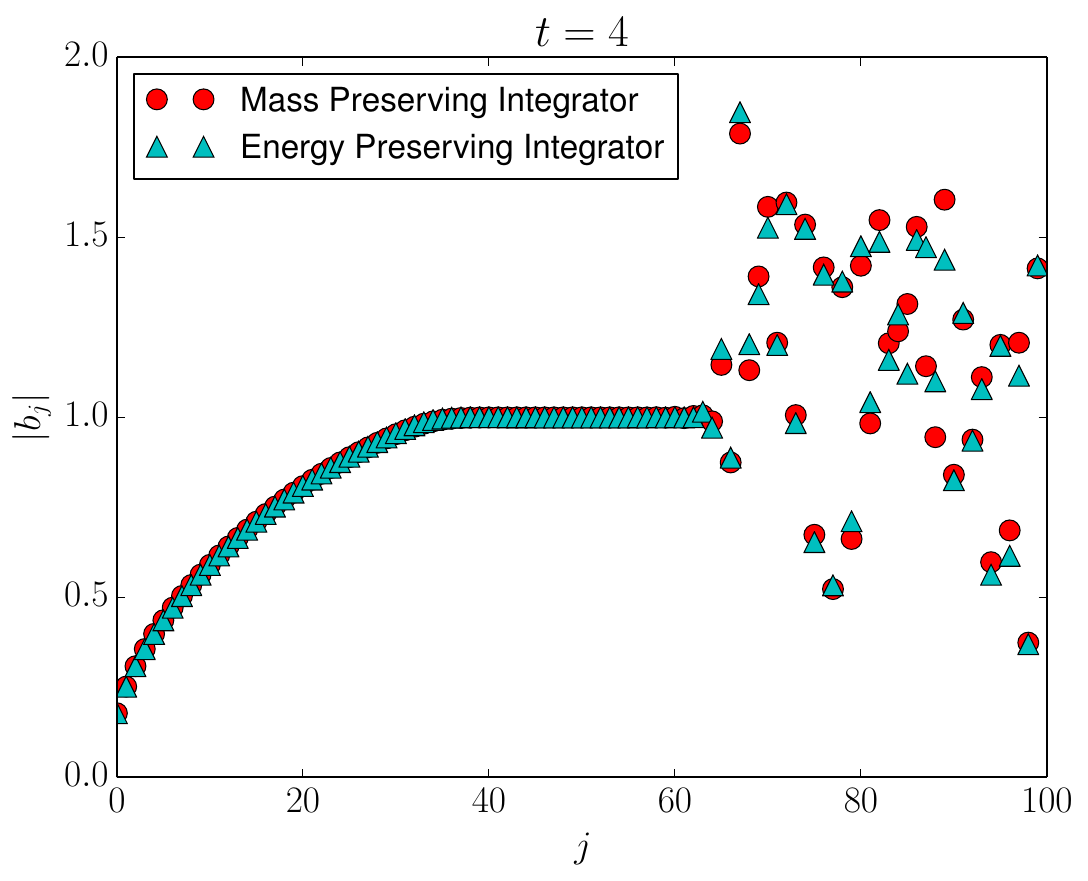}}
  \caption{The integrators reproduce the rarefactive and dispersive
    wave-like behavior for initial condition \eqref{e:ic_shock}
    observed in \cite{Colliander:2013hz, herr2013discrete} and show
    agreement with one another.  Computed with $\dt = 0.1$.}
  \label{f:shock1}
\end{figure}

\begin{figure}

  \includegraphics[width=8cm]{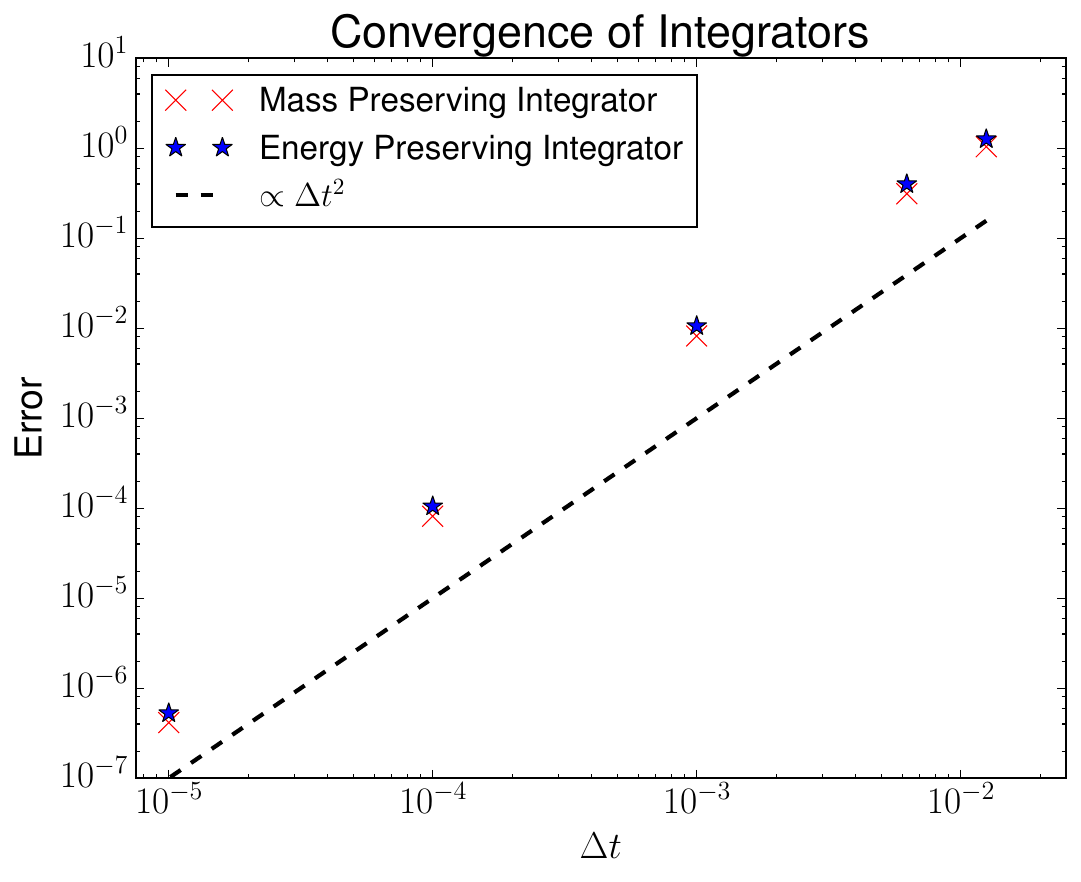}
  \caption{The error of integrators \eqref{e:mass_scheme} and
    \eqref{e:eng_scheme} compared to a high quality RK4 solution
    computed with $\Delta t=10^{-4}$, \eqref{e:error1}.}
  \label{f:convergence}
\end{figure}

\begin{table}
  \caption{Error, \eqref{e:error1}, computed using the
    surrogate RK4 solution.}
  \label{t:err1}
  \begin{tabular}{ |l|c|c|c|c|c| } 
    \hline  
    $\Delta t$& Mass & Energy & Trapezoidal& Implicit Midpoint & Projection \\
    \hline\hline 
    0.1 & 0.18 & 0.20 & 0.19 & 0.20 & 0.11 \\
    \hline 
    0.05 & 0.06 & 0.07 & 0.07 & 0.07 & 0.02 \\
    \hline
    0.025 & 0.02 & 0.02 & 0.02 & 0.02 & 2.32$\cdot10^{-3}$ \\
    \hline
    0.0125 & 5.05$\cdot10^{-3}$ & 5.85$\cdot10^{-3}$ & 5.56$\cdot10^{-3}$ & 5.56$\cdot10^{-3}$ & 3.06$\cdot10^{-4}$ \\
    \hline
  \end{tabular}
\end{table}

\begin{table}
  \caption{Relative error in the mass invariant, $\max_n |\calM[\bb_n]-\calM|/\calM$.}
  \label{t:mass1}
  \begin{tabular}{ |l|c|c|c|c|c| } 
    \hline  
    $\Delta t$& Mass & Energy & Trapezoidal& Implicit Midpoint & Projection \\
    \hline\hline 
    0.1 & 2.84$\cdot10^{-16}$ & 1.59$\cdot10^{-4}$ & 3.82$\cdot10^{-4}$ & 7.11$\cdot10^{-16}$ & 1.99$\cdot10^{-15}$ \\
    \hline 
    0.05 & 4.26$\cdot10^{-16}$ & 3.87$\cdot10^{-5}$ & 1.53$\cdot10^{-4}$ & 2.84$\cdot10^{-16}$ & 2.13$\cdot10^{-15}$ \\
    \hline
    0.025 & 7.11$\cdot10^{-16}$ & 9.60$\cdot10^{-6}$ & 4.69$\cdot10^{-5}$ & 5.68$\cdot10^{-16}$ & 3.41$\cdot10^{-15}$ \\
    \hline
    0.0125 & 5.68$\cdot10^{-16}$ & 2.39$\cdot10^{-6}$ & 1.26$\cdot10^{-5}$ & 4.26$\cdot10^{-16}$ & 6.54$\cdot10^{-15}$ \\
    \hline
  \end{tabular}
\end{table}

\begin{table}
  \caption{Relative error in the energy invariant, $\max_n |\calH[\bb_n]-\calH|/\calH$.}
  \label{t:eng1}
  \begin{tabular}{ |l|c|c|c|c|c| } 
    \hline  
    $\Delta t$& Mass & Energy & Trapezoidal& Implicit Midpoint & Projection \\
    \hline\hline 
    0.1 & 9.33$\cdot10^{-4}$ & 1.85$\cdot10^{-15}$ & 4.43$\cdot10^{-3}$ & 2.51$\cdot10^{-3}$ & 5.54$\cdot10^{-15}$ \\
    \hline 
    0.05 & 4.87$\cdot10^{-4}$ & 1.71$\cdot10^{-15}$ & 2.07$\cdot10^{-3}$ & 1.09$\cdot10^{-3}$ & 3.68$\cdot10^{-14}$\\
    \hline
    0.025 & 1.75$\cdot10^{-4}$ & 1.85$\cdot10^{-15}$ & 6.96$\cdot10^{-4}$ & 3.53$\cdot10^{-4}$ & 1.14$\cdot10^{-15}$ \\
    \hline
    0.0125 & 5.00$\cdot10^{-5}$ & 1.85$\cdot10^{-15}$ & 1.95$\cdot10^{-4}$ & 9.77$\cdot10^{-5}$ & 2.88$\cdot10^{-14}$\\
    \hline
  \end{tabular}
\end{table}

\subsection{Performance of Nonlinear Solvers}
\label{s:newton}

As our proposed methods require solving a nonlinear system at each
time step, the number of required Newton iterations bears
consideration.  Here, we repeat the numerical experiment of Section
\ref{s:pointwise}, solving \eqref{e:ic_shock} with $N=100$, and
integrate out to $t_{\max}=1$.  The results, with the same tolerances,
are given in Tables \ref{t:feval1} and \ref{t:linear1}. The Mass,
Energy, Trapezoidal Rule, and Implicit Midpoint solvers had comparable
performance, with no discernible advantages.  Between four and six
function evaluations are needed for these solvers per time step for
each of these.  In contrast, the Projection method typically took
fewer function evaluations, but requires the additional four function
evaluations from RK4.  Thus, as measured by the number of function
evaluations, these solvers, and RK4, are quite comparable.

A modest number of between three and five iterations of the Newton
solver are needed for the implicit solvers, while the projection
method typically requiring only two to three.  The Projection method
has the significant advantage of requiring a much smaller system to be
solved.  For all of solvers, analytic Jacobians were provided which
offered a significant reduction in the number of function evaluations.



\begin{table}
  \caption{Average number of function evaluations
    needed per time step.}
  \label{t:feval1}
  \begin{tabular}{ |l|c|c|c|c|c| } 
    \hline  
    $\Delta t$& Mass & Energy & Trapezoidal& Implicit Midpoint & Projection \\
    \hline\hline 
    0.1 & 5.00 & 5.00 & 5.00 & 5.00 & 4.00\\
    \hline 
    0.05 & 4.62 & 5.00 & 5.70 & 5.70 & 3.35\\
    \hline
    0.025 & 5.00 & 5.00 & 5.00 & 5.00 & 3.00\\
    \hline
    0.0125 & 5.00 & 5.00 & 5.00 & 5.00 & 3.00\\
    \hline
  \end{tabular}
\end{table}

\begin{table}
  \caption{Average number of Newton iterations
    needed per time step.}
  \label{t:linear1}
  \begin{tabular}{ |l|c|c|c|c|c| } 
    \hline  
    $\Delta t$& Mass & Energy & Trapezoidal& Implicit Midpoint & Projection \\
    \hline\hline 
    0.1 & 4.00 & 4.00 & 4.00 & 4.00 & 3.00\\
    \hline 
    0.05 & 3.62 & 4.00 & 4.70 & 4.70 & 2.35\\
    \hline
    0.025 & 4.00 & 4.00 & 4.00 & 4.00 & 2.00\\
    \hline
    0.0125 & 4.00 & 4.00 & 4.00 & 4.00 & 2.00\\
    \hline
  \end{tabular}
\end{table}




\subsection{Ensemble Simulations and Weak Turbulence}

One way to measure energy transfer in \eqref{e:toy} is through the
Sobolev type $h^s$ norm
\begin{equation}
  \label{e:hsnorm}
  \|\bb(t)\|_{h^s}^2 = \sum_{j=1}^N 2^{(s-1)j} |b_j(t)|^2.
\end{equation}
This is closely related to the measurements for energy transfer in
\cite{Colliander:2010wz,Colliander:2013hz}.  For a single initial
condition, we may find that $h^s$ norm grows in time, but for a
phenomenon to constitute weak turbulence, we would expect such a
transfer to be generic.  Thus, we simulate an ensemble of initial
conditions and examine the average evolution of \eqref{e:hsnorm}.

Our ensemble of initial conditions are constructed as follows.  For
the $k$-th sample,
\begin{equation}
  \label{e:ensemble_ic}
  b_j^{(k)}(0) = \frac{1}{4^{j-1}}\cdot \exp\{i \theta_j^{(k)}\}, \quad j=1,\ldots,N,
\end{equation}
where $ \theta_j^{(k)}\sim U(0, 2\pi)$ are independently and
identically distributed.  The purpose of the decay in
\eqref{e:ensemble_ic} is such that when we look at the large $N$
limit, the $h^s$ norms remain finite for $s\leq 4$.  Our random phases
were generated in parallel using the Scalable Parallel Random Number
Generator (SPRNG) 2.0, \cite{mascagni2000algorithm}, available at
\url{http://www.sprng.org}.  Our sample size in each case was one
hundred.

The results of several of our simulations are shown Figures
\ref{f:turb1} and \ref{f:turb2} where we plot the ensemble averaged
evolution of the $h^s$ norm,
\begin{equation}
  \label{e:avgnorm}
  \mean{\norm{\bb(t)}_{h^s}} = \frac{1}{M}\sum_{k=1}^M \norm{\bb^{(k)}(t)}_{h^s}
\end{equation}
for different integrators using different time steps.  First, in all
cases, there is a generic tendency for the norms to grow.  Second, as
shown in Figure \ref{f:turb1}, the conservative schemes show
consistent growth rates, independent of time step, out to
$t_{\max}=1000$.  One notable feature is that modified midpoint mass
preserving integrator, \eqref{e:mass_scheme}, appears to have a
comparatively larger variance than the other symmetric methods.

In Figure \ref{f:turb2}, we compare the conservative integrators to
each other, along with Trapezoidal Rule and RK4.  Again, there is
consistency in the the ensemble averaged behavior.  While we have only
plotted the results for $s=4$, this is consistent with the other norms
we have examined.

\begin{figure}

  \subfigure{\includegraphics[width=6.25cm]{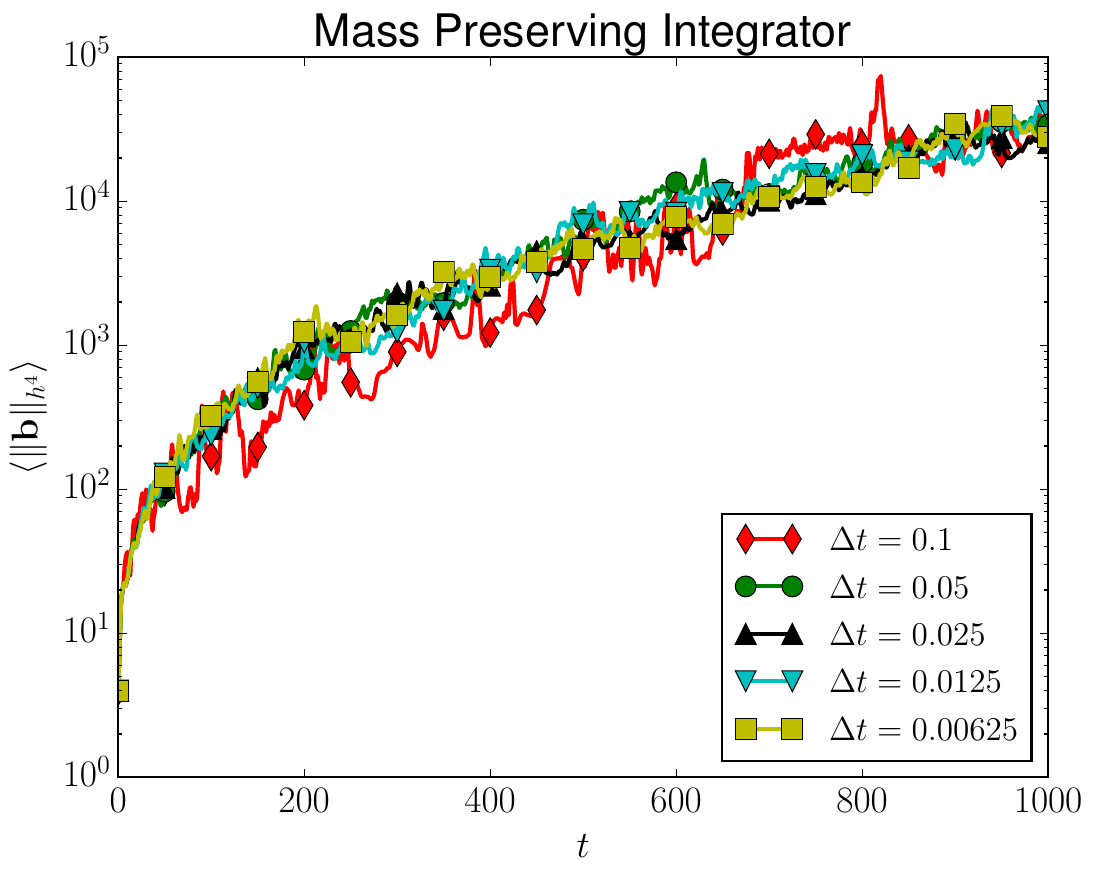}}
  \subfigure{\includegraphics[width=6.25cm]{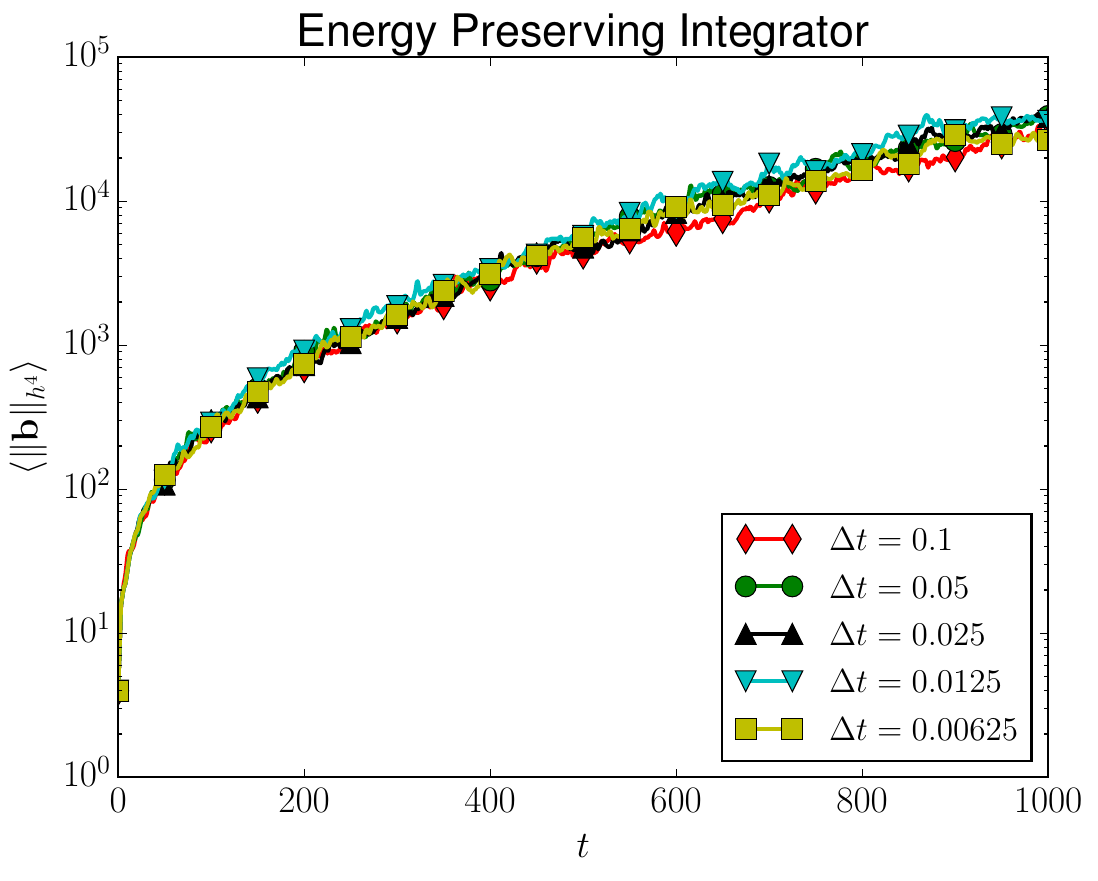}}
  \caption{Ensemble averaged norm evolution for
    two conservative integrators and different values of $\dt$.
    In all cases, the ensemble averages were consistent as a function
    of time.}
  \label{f:turb1}
\end{figure}

\begin{figure}

  \subfigure{\includegraphics[width=6.25cm]{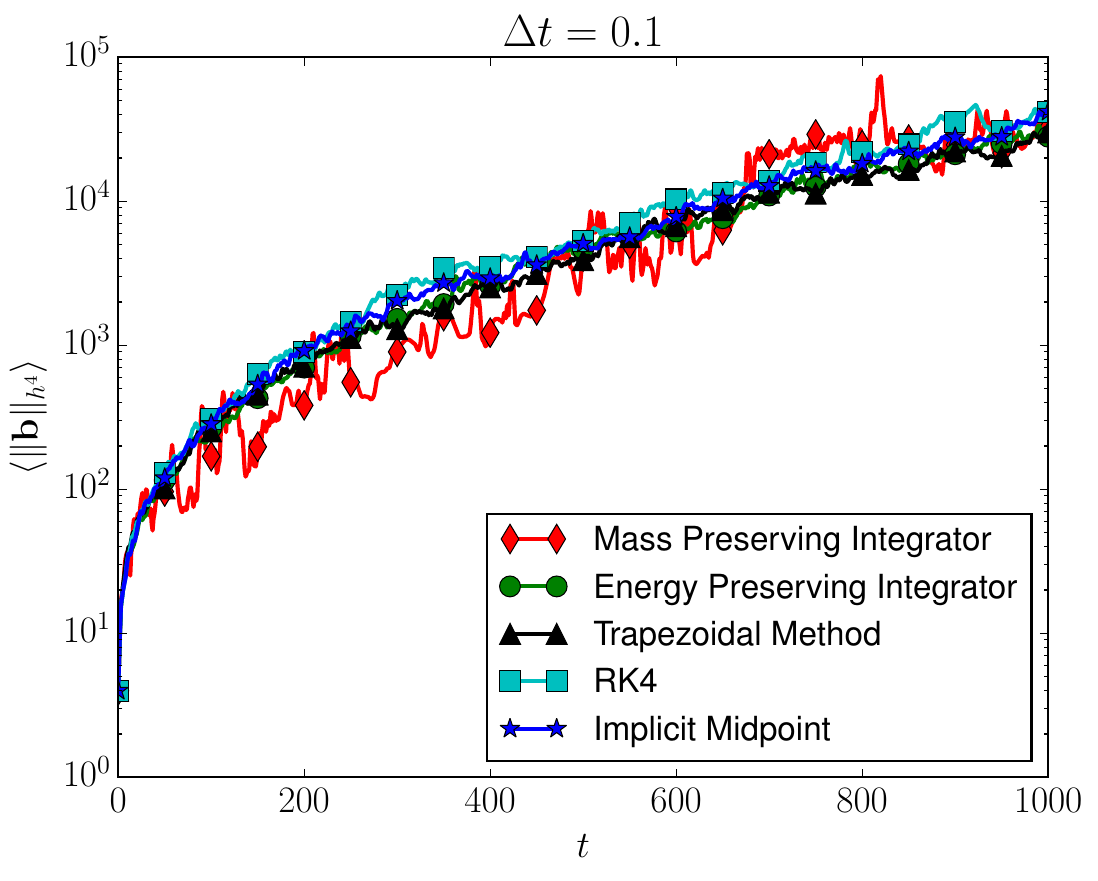}}
  \subfigure{\includegraphics[width=6.25cm]{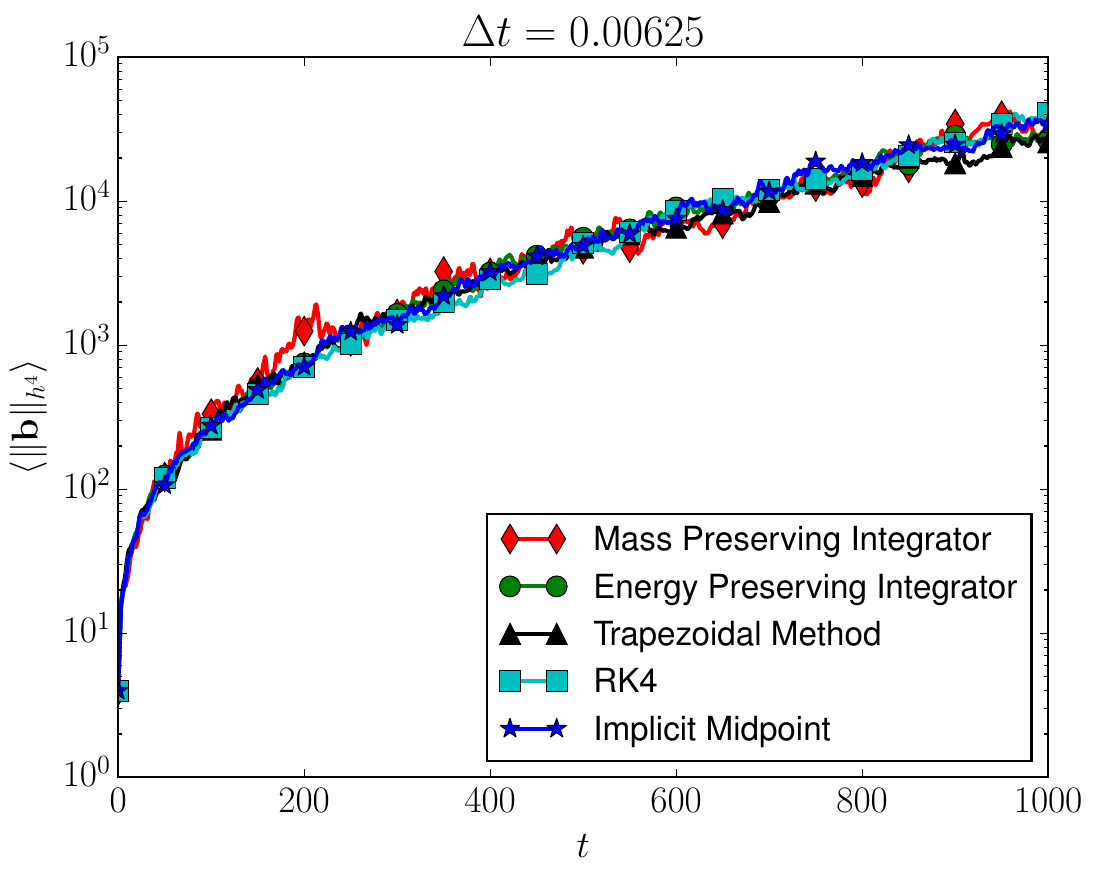}}
  \caption{Ensemble averaged norm evolution for
   several integrators.  On this time scale, the
    four methods are consistent at both values of $\dt$.}
  \label{f:turb2}
\end{figure}

On longer time scales, we see the advantage of our conservative
integrators.  Integrating out to $t_{\max}=10^5$, we see a systematic
bias in the norm, shown in Figure \ref{f:turb3}.  To better understand
this, we examine the ensemble averaged energy and mass invariants for
the four methods.  These are shown in Figure \ref{f:invarianterror}.
Our conservative integrators and trapezoidal rule behave well, while
the errors in mass and energy continue to grow when the RK4 method is
used.  Thus, on long time scales, fixed step Runge-Kutta methods will
give biased results.

The boundedness of the error in the energy of the symmetric schemes
which do not conserve energy, is unsurprising.  In particular,
implicit midpoint, being symplectic, will conserve some modified
Hamiltonian, $\tilde H$, which will be nearly preserved over very long
periods of integration and converge to $H$ as $\dt \to 0$,
\cite{Hairer:2006vg}.

\begin{figure}
  \includegraphics[width=8cm]{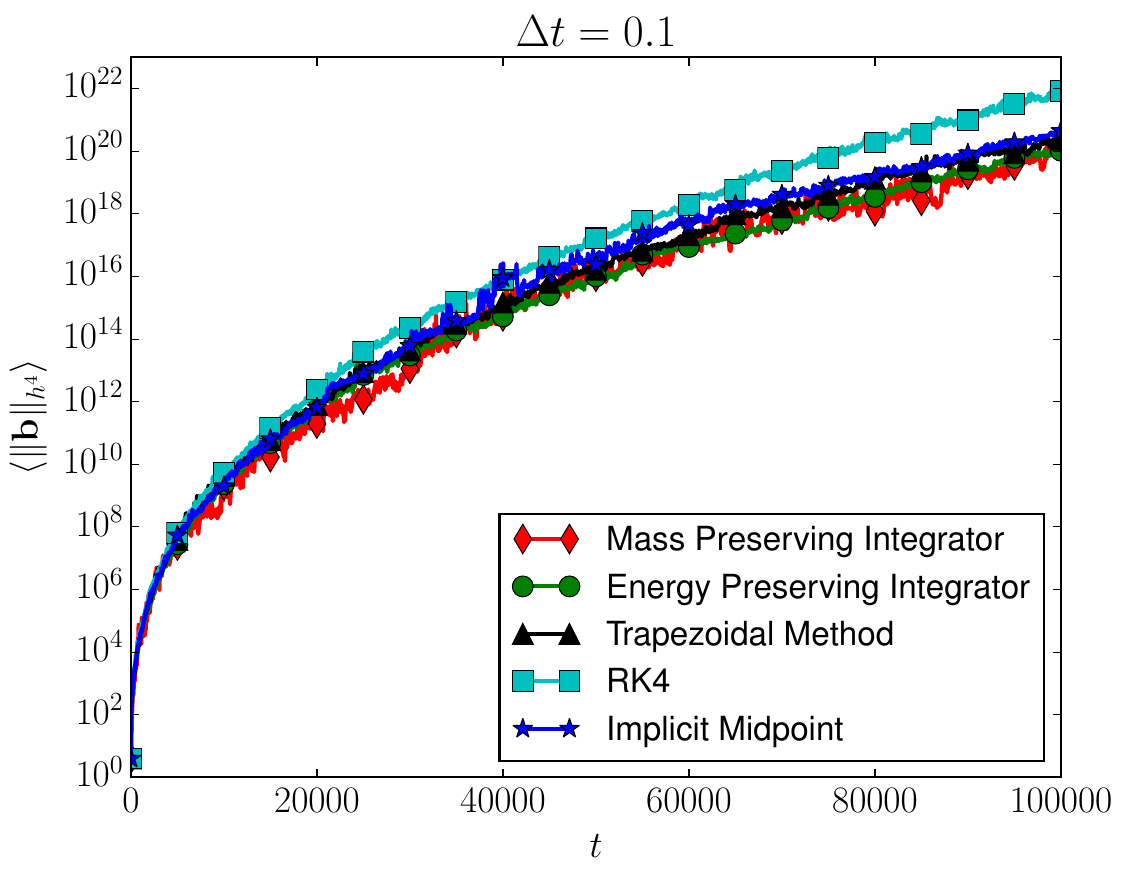}
  \caption{Ensemble averaged norm evolution of several of the
    integrators.  On longer time
    scales, a large systematic bias appears in the RK4 Solution.}
  \label{f:turb3}
\end{figure}

\begin{figure}
  \subfigure{\includegraphics[width=6.25cm]{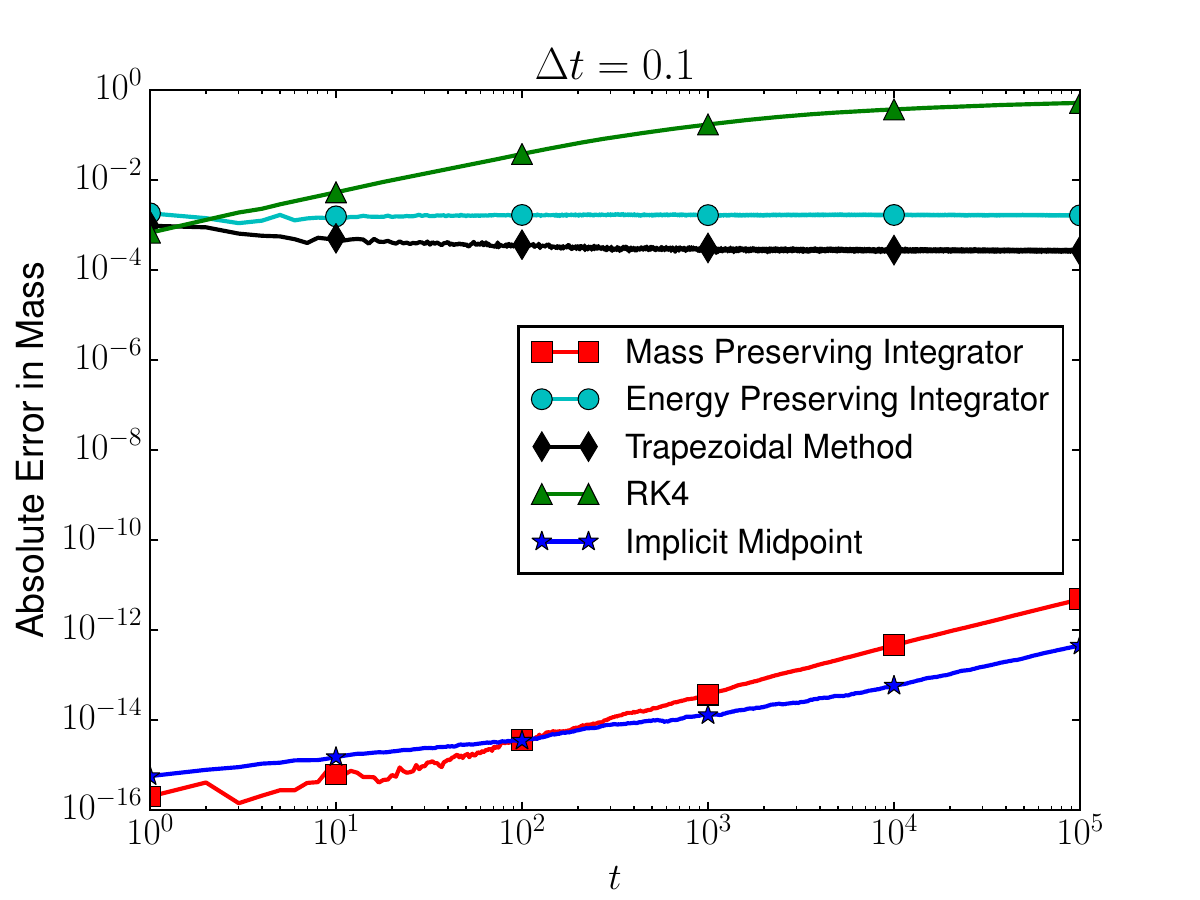}}
  \subfigure{\includegraphics[width=6.25cm]{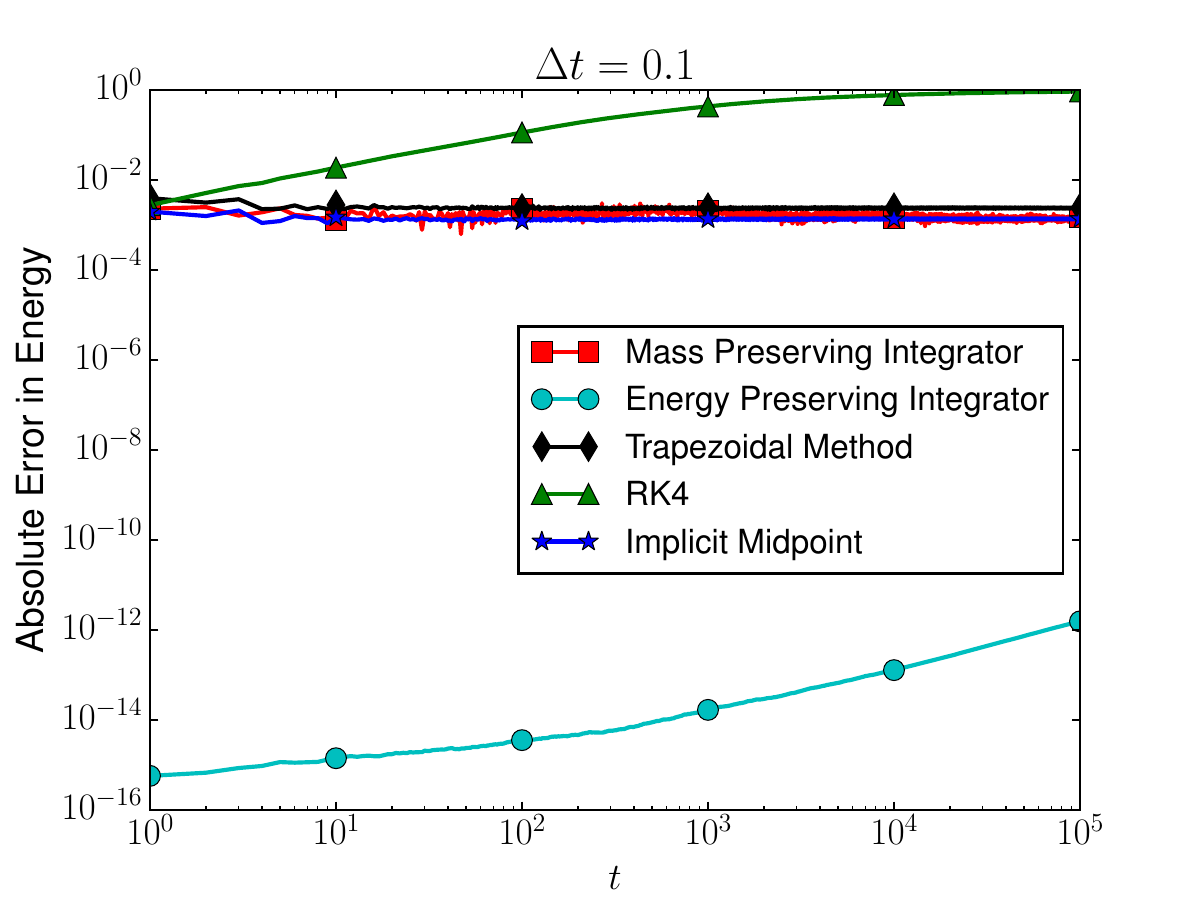}}
  \caption{The conservative integrators accurately preserve the
    appropriate invariants of the system. 
    Trapezoidal Rule is more consistent than RK4, which
    systematically adds energy to the system.}
  \label{f:invarianterror}
\end{figure}

\section{Discussion}
\label{s:discussion}

We have formulated integrators which conserve the invariants of the
Toy Model System, \eqref{e:toy}.  The local truncation error in both
cases is second order, and they provide robust behavior in
simulations.  However, we were only able to prove convergence of
schemes which conserve mass, as we needed an {\it a priori} bound on
the solution.  One outstanding question is thus to prove convergence
of the energy preserving scheme.  Another is to produce a method that
intrinsically preserves both invariants, without projection.

In comparison to other methods, these schemes are quite favorable,
both in terms of their properties and computational cost.  For large
scale, long time, statistical studies, they will inevitably perform
better than fixed step Runge-Kutta methods, though adaptive RK methods
may outperform them.

\bibliographystyle{abbrvnat}

\bibliography{toy}

\end{document}